\documentclass[leqno,12pt]{amsart}

\usepackage{amssymb}
\usepackage{amscd}              

\newcommand{\ad}{\mathrm{ad}}

\newcommand{\supp}{\mathrm{Supp}}
\newcommand{\Hilb}{\mathrm{Hilb}}

\newtheorem{theorem}{Theorem}
\newtheorem{theoremdefinition}{Theorem and Definition}
\newtheorem{proposition}[theorem]{Proposition}
\newtheorem{lemma}[theorem]{Lemma}
\newtheorem{corollary}[theorem]{Corollary}

\theoremstyle{definition}
\newtheorem{definition}[theorem]{Definition}
\newtheorem{remark}[theorem]{Remark}

\begin{document}

\title{{Invariant Hilbert schemes and wonderful varieties}}

\author{St\'ephanie Cupit-Foutou}
\email{scupit@math.uni-koeln.de}
\address{Universit\"at zu K\"oln, Mathematisches Institut, Weyertal 78-90, 50978 K\"oln, Germany}

\begin{abstract}
The invariant Hilbert schemes considered in \cite{Js} and \cite{BC1} were proved to be smooth.
The proof relied on the classification of strict wonderful varieties.
We obtain in the present article a classification-free proof of the smoothness of these invariant Hilbert schemes
by means of deformation theoretical arguments. 
As a consequence we recover in a shorter way (case-by-case-free considerations) the classification of strict wonderful varieties.
This provides an alternative approach to Luna's conjecture.
\end{abstract}

\maketitle

\section{Introduction}
Let $G$ be a connected reductive complex algebraic group.
Subschemes of a given finite dimensional $G$-module, whose coordinate rings are isomorphic as $G$-modules
are parameterized by a quasiprojective scheme, the so-called \textsl{invariant Hilbert scheme}
introduced by Alexeev and Brion in~\cite{AB}.

Let us consider affine multi-cones over some flag $G$-varieties which are $G$-orbit closures. 
Invariant Hilbert schemes which contain such a multi-cone as closed point are proved to be smooth in~\cite{Js} and \cite{BC1}. 
Further, their universal family can be realized as the $G$-orbit map of the normalization of a certain multi-cone over a \textsl{ strict wonderful variety}.

Wonderful $G$-varieties are projective $G$-varieties which enjoy nice properties like being smooth and having a dense orbit for a Borel subgroup of $G$
(e.g. flag varieties, De Concini-Procesi compactifications of symmetric spaces).
Luna's conjecture asserts that wonderful varieties are classified by combinatorial objects called spherical systems.
In~\cite{BC2}, we answered this conjecture positively in case of wonderful varieties whose points have selfnormalising stabilizer: the so-called \textsl{strict wonderful varieties};
like all (positive) answers obtained so far (\cite{L01,BP,B}), 
we followed Luna's Lie theoretical approach introduced in~\cite{L01} which involves case-by-case considerations.

The classification of strict wonderful varieties is one of the main tools which allowed to describe the invariant Hilbert scheme in~\cite{Js} and \cite{BC1}.

In this work, we obtain a classification-free proof of the smoothness of the invariant Hilbert scheme under study
by means of deformation theoretical arguments. 
As a consequence we recover in a shorter way (case-by-case-free considerations) the classification of strict wonderful varieties in terms of spherical systems.
More specifically, we provide a geometrical construction of these wonderful varieties; this yields an alternative approach to Luna's conjecture.

\bigbreak

\paragraph{\textit{Acknowledgements}}
I am indebted and grateful to Michel Brion for his suggestions and helpful discussions.
I thank Paolo Bravi and S\'ebastien Jansou for interesting exchanges.

\bigbreak
\smallbreak\noindent{\textbf{Main notation.}}
The ground field is the field of complex numbers.
Throughout this paper, $G$ is a connected reductive algebraic group whose Lie algebra is denoted by $\mathfrak g$.
Let $B$ be a Borel subgroup of $G$, $U$ its unipotent radical and $T\subset B$ a maximal torus of $G$.
The set of dominant weights $\Lambda^+$ (relatively to $B$ and $T$) parameterizes the irreducible $G$-modules;
for a given $\lambda\in\Lambda^+$, we shall denote by $V(\lambda)$ the corresponding $G$-module.
The $\mu$-weight space ($\mu$ being any character of $T$) of a given $G$-module $V$ is denoted by $V_\mu$; in particular $V(\lambda)_\lambda$ is one-dimensional.

\section{Invariant Hilbert schemes}

\subsection{Definitions and setting}

We recall from~\cite{AB} notions and results concerning invariant Hilbert schemes.

Given a finite dimensional $G$-module $V$ and a scheme $S$ endowed with the trivial action of $G$, 
a closed $G$-subscheme $\mathfrak X$ of  $V\times S$ is called a \emph{family of affine $G$-subschemes of $V$ over $S$}.

The projection of $V\times S$ onto $S$ induces a morphism $\pi:\mathfrak X\rightarrow S$ which is affine, of finite type and $G$-equivariant.
We thus have a $G$-equivariant morphism of $\mathcal O_S-G$-modules
\begin{equation}
\pi_\star \mathcal O_{\mathfrak X}\cong
\bigoplus_{\lambda\in\Gamma}\mathcal R_\lambda\otimes V(\lambda)^\ast
\end{equation}
where $\Gamma$ denotes a subset of $\Lambda^+$ and $\mathcal R_\lambda$ is the $U$-fixed point set 
$(\pi_\star \mathcal O_\mathfrak X)_\lambda^U$.

Each $\mathcal R_\lambda$ is a coherent sheaf of $(\pi_\star\mathcal O_\mathfrak X)^G$-modules.
When each $\mathcal R_\lambda$ is an invertible sheaf of $\mathcal O_S$-modules,
the morphism $\pi$ is flat and the family $\mathfrak X$ is said to be \textsl{of type $\Gamma$}.

\begin{theoremdefinition}(\cite[Theorem~1.7]{AB})~\label{definitionHilb}
The contravariant functor 
that associates to any scheme $S$ the set of families of affine $G$-subschemes of $V$ over $S$ of type $\Gamma$ is
represented by a quasi-projective scheme, the \textsl{invariant Hilbert scheme} $\Hilb_\Gamma^G$. 
\end{theoremdefinition}

We shall be concerned throughout this article by the case where  
$$
V=V(\lambda_1)\oplus\ldots\oplus V(\lambda_s)
$$
with
$\lambda_1,\ldots,\lambda_s$ linearly independent dominant weights spanning a free and \textsl{saturated} monoid $\Gamma$, \textsl{i.e.} $\Gamma$ is such that
$$
\mathbb Z\Gamma\cap\Lambda^+=\Gamma.
$$

In the remaining, we shall denote
$$
\underline\lambda=(\lambda_1,\ldots,\lambda_s).
$$

Let $X_0$ be the $G$-orbit closure within $V$ of 
$$
v_{\underline\lambda}=v_{\lambda_1}+\ldots+v_{\lambda_s}
$$
where $v_{\lambda_i}$ denotes a highest weight vector of $V(\lambda_i)$ for $i=1,\ldots,s$.

Under the saturation assumption, the variety $X_0$ is spherical under the action of $G$, 
that is $X_0$ contains a dense orbit for a Borel subgroup of $G$.
Further, $X_0$ is normal and the boundary $X_0\setminus G.v_{\underline\lambda}$ is of codimension greater than $2$. 

Let $G_{v_{\underline\lambda}}$ be the stabilizer of $v_{\underline\lambda}$ in $G$ 
and $P_{\underline\lambda}$ be the normalizer of $G_{v_{\underline\lambda}}$ in $G$. 
The variety $X_0$ thus coincides with the affine multi-cone
$$
\mathrm{Spec}\oplus_{\underline\nu\in\Gamma} H^0(G/P_{\underline\lambda},\mathcal L_{\underline\nu})
$$ 
where $\mathcal L_{\underline\nu}$ refers to the $G$-linearized bundle $\otimes_i\mathcal L_{\lambda_i}^{m_i}$ with $\underline\nu=\sum_{i=1}^{s} m_i\lambda_i$ a dominant weight and
$$
\mathcal L_\lambda=G\times_B\mathbb C_{-\lambda}
$$
where $\mathbb C_{-\lambda}$ stands for the one-dimensional $B$-module associated to the character $-\lambda$.

The subvariety $X_0$ of $V$ can thus be regarded as a closed point of $\Hilb_\Gamma^G$.
More generally, any closed point of $\Hilb_\Gamma^G$ having a multiplicity-free coordinate ring,
is a spherical affine $G$-variety; \textsl{see}\cite{Br1,Vi}.
Such an affine $G$-variety is said to be non-degenerate if its projection onto any $V(\lambda_i)$, $i=1,\ldots,s$, is not trivial.

\begin{theorem}(\cite[Corollary~1.17 and Theorem 2.7]{AB})
The non-degenerate $G$-subvarieties of $V$ which can be seen as closed points of $\Hilb_\Gamma^G$ are parameterised by
a connected and open subscheme $\Hilb_{\underline\lambda}^G$ of $\Hilb_\Gamma^G$. 
\end{theorem}

\subsection{Tangent space of the invariant Hilbert scheme}

Let $T_\ad$ be the adjoint torus of $G$, that is $T_\ad=T/Z(G)$ where $Z(G)$ is the center of $G$.
\textit{Any} invariant Hilbert scheme is endowed with an action of the adjoint torus (\textit{see}~\cite{AB} for details).

Let us recall how $T_\ad$ acts on the tangent space $T_{X_0}\Hilb_\Gamma^G$ at $X_0$ of the invariant Hilbert scheme $\Hilb_\Gamma^G$ (\textsl{see} Section 2.1 in~\cite{AB}). 
Let $t\in T_\ad$ then
$$
t.v=(\lambda_i-\mu)(t)v \quad \mbox{ when $v\in V(\lambda_i)_\mu$}.
$$

\begin{theorem}(\cite[Proposition 1.5]{AB})
The tangent space $T_{X_0}\Hilb_\Gamma^G$ at $X_0$ of $\Hilb_\Gamma^G$
is isomorphic as a $T_\ad$-module to the $G_{v_{\underline\lambda}}$-invariant subspace $\left(V/\mathfrak g.v_{\underline\lambda}\right)^{G_{v_{\underline\lambda}}}$.  
\end{theorem}

\begin{theorem}(\cite[Theorem 2.2, Theorem 4.1]{BC1} and also \cite{Js} if $s=1$)\label{tgtspaceasmodule}
The tangent space $T_{X_0}\Hilb_\Gamma^G$ is a multiplicity free $T_{\ad}$-module.
Further, its $T_{\ad}$-weights belong to Table~\ref{sphericalroots}.
\end{theorem}

\begin{table}[tbh]
\caption{$T_{\ad}$-weights in $(V/\mathfrak g.v_{\underline\lambda})^{G_{v_{\underline\lambda}}}$.}
\label{sphericalroots}
\begin{center}
\begin{tabular}{l|l}
\hline
Type of support & \multicolumn{1}{|c}{Weight} \\
\hline
$\mathsf A_1\times\mathsf A_1$ & $\alpha+\alpha'$ \\
\hline
$\mathsf A_n$ & $\alpha_1+\ldots+\alpha_n$, $n\geq2$ \\
& $2\alpha$, $n=1$ \\
& $\alpha_1+2\alpha_2+\alpha_3$, $n=3$ \\
\hline
$\mathsf B_n$, $n\geq2$ & $\alpha_1+\ldots+\alpha_n$ \\
& $2\alpha_1+\ldots+2\alpha_n$ \\
& $\alpha_1+2\alpha_2+3\alpha_3$, $n=3$ \\
\hline
$\mathsf C_n$, $n\geq3$ & $\alpha_1+2\alpha_2+\ldots+2\alpha_{n-1}+\alpha_n$ \\
\hline
$\mathsf D_n$, $n\geq4$ & $2\alpha_1+\ldots+2\alpha_{n-2}+\alpha_{n-1}+\alpha_n$ \\
& $\alpha_1+2\alpha_2+2\alpha_3+\alpha_4$, $n=4$ \\
& $\alpha_1+2\alpha_2+\alpha_3+2\alpha_4$, $n=4$ \\
\hline
$\mathsf F_4$ & $\alpha_1+2\alpha_2+3\alpha_3+2\alpha_4$ \\
\hline
$\mathsf G_2$ & $4\alpha_1+2\alpha_2$ \\
& $\alpha_1+\alpha_2$ \\
\hline
\end{tabular}
\end{center}
\end{table}
The weights occurring in Table~\ref{sphericalroots} are so-called \textsl{spherical roots}; they are closely related to wonderful varieties (\textsl{see}~\cite{L97,L01} or~Section~\ref{connection} where we have reported from loc. cit. the results we shall need).

In Proposition~\ref{non-rigidity} and Proposition~\ref{converse}, we will characterize in a purely combinatorial way the monoids $\Gamma$ such that $T_{X_0}\Hilb_\Gamma^G$ is not trivial.

\subsection{Obstruction space and smoothness}
Let us recall first from~\cite{Sc} (\textsl{see} also~\cite{Se}) the definition and some properties of the second cotangent module $T^2_{X_0}$ of $X_0$.

Let $I\subset\mathrm{Sym}(V^*)$ be the ideal of the affine multi-cone $X_0\subset V$
and let $\mathcal O(X_0)=\mathrm{Sym}(V^*)/I$.
Take a presentation of $I$ as $\mathcal O(X_0)$-modules
$$
0\rightarrow R\rightarrow F\rightarrow I\rightarrow 0
$$
where $F$ is a finitely generated free $\mathcal O(X_0)$-module.

This induces an exact sequence of $\mathcal O(X_0)$-modules

\begin{equation}\label{presentation}
R/K\rightarrow F\otimes \mathcal O(X_0)\rightarrow I/I^2\rightarrow 0
\end{equation}
where $K$ is the module of Koszul relations.

Let us apply $\mathrm{Hom}(-,\mathcal O(X_0)):=\mathrm{Hom}_{\mathcal O(X_0)}(-,\mathcal O(X_0) )$ to the exact sequence~(\ref{presentation})
then $T^2_{X_0}$ is defined by the exact sequence 
$$
\mathrm{Hom}(I/I^2,\mathcal O_{X_0})\rightarrow \mathrm{Hom}(F\otimes \mathcal O(X_0),\mathcal O_{X_0})\rightarrow \mathrm{Hom}( R/K,\mathcal O_{X_0})\rightarrow T^2_{X_0}\rightarrow 0.
$$

The $\mathcal O(X_0)$-module $T^2_{X_0}$ is independent of the presentation of the ideal $I$ of $X_0$.
Since $X_0$ is normal and $X_0\setminus G.v_{\underline\lambda}$ is of codimension greater or equal to $2$,
we have:
\begin{equation}\label{1srcharacterisationof T^2}
0\rightarrow T^2_{X_0}\rightarrow
H^1(G.v_{\underline\lambda},\mathcal N_{\underline\lambda})\rightarrow 
 H^1(G.v_{\underline\lambda},F\otimes\mathcal O_{X_0}),
\end{equation}
where $\mathcal N_{\underline\lambda}$ stands for the normal sheaf of $G.v_{\underline\lambda}$ in $V$.

By Schlessinger´s Comparison Theorem (\textsl{see}~\cite{Sc} and also Section 2.3 and Proposition 3.1.12 in~\cite{Se},
we have

\begin{proposition}\label{Schlessinger}
If the invariant set $(T^2_{X_0})^G$ is trivial then the invariant Hilbert scheme $\Hilb^G_\Gamma$ is smooth at $X_0$.
\end{proposition}

\begin{theorem}\label{smoothness}
The invariant Hilbert scheme $\Hilb_{\underline\lambda}^G$ is smooth at $X_0$.
\end{theorem}

\begin{corollary}\label{affinity}
The invariant Hilbert scheme $\Hilb_{\underline\lambda}^G$ is an affine space.
\end{corollary}

\begin{proof}
Being smooth and connected, $\Hilb_{\underline\lambda}^G$ is irreducible.
By Corollary 3.4 in \cite{AB}, the invariant Hilbert scheme is acted on by the adjoint torus $T_{ad}$ of $G$ with finitely many $T_{\ad}$-orbits.
Hence it is a toric variety for the adjoint torus of $G$.
Further, a smooth affine toric variety with a single fixed point being an affine space, $\Hilb_{\underline\lambda}^G$ is in turn an affine space. 
\end{proof}

The proof of the above theorem is conducted in details in Section~\ref{Proofs}.
By Proposition~\ref{Schlessinger}, 
it amounts to proving that $(T^2_{X_0})^G$ is trivial;
this will be achieved by means of the two following propositions.

For any $1\leq i\neq j\leq s$, let
$$
K_{i,j}=\ker\{V(\lambda_i)\otimes V(\lambda_j)\longrightarrow V(\lambda_i+\lambda_j)\}
$$
and 
$$
K_{i,i}=\ker\{\mathop{S^2}V(\lambda_i)\longrightarrow V(2\lambda_i)\}.
$$
This yields the presentation
$$
0\rightarrow R\rightarrow \oplus_{1\leq i,j\leq s} K_{i,j}\otimes\mathcal O(X_0)\rightarrow I/I^2\rightarrow 0.
$$

\begin{proposition}\label{characterisation}
Let $\mathfrak g_{v_{\underline\lambda}}$ be the isotropy Lie algebra of $v_{\underline\lambda}$
and $G_{v_{\underline\lambda}}^\circ$ be the identity component of $G_{v_{\underline\lambda}}$.
The invariant set $(T^2_{X_0})^G$ is given by the kernel of the map
$$
H^1(\mathfrak g_{v_{\underline\lambda}},V/\mathfrak g.v_{\underline\lambda})^{G_{v_{\underline\lambda}}/G_{v_{\underline\lambda}}^\circ}\longrightarrow 
\oplus_{1\leq i,j\leq s} H^1(\mathfrak g_{v_{\underline\lambda}},K_{i,j})^{G_{v_{\underline\lambda}}/G_{v_{\underline\lambda}}^\circ}
$$
induced by the map of $\mathfrak g_{v_{\underline\lambda}}$-modules
$$
v\mapsto \sum_i v\cdot v_{\lambda_i}\quad\mbox{ where $v \in V$ and } v\cdot v_{\lambda_i}\in V\cdot V(\lambda_i).
$$
\smallbreak\noindent
\end{proposition}

\begin{proof}
The sheaf $\mathcal N_{\underline\lambda}$
being the $G$-linearized sheaf on $G/G_{v_{\underline\lambda}}$ associated to the $G_{v_{\underline\lambda}}$-module $V/\mathfrak g.v_{\underline\lambda}$,
we have: 
$$
H^1(G.v_{\underline\lambda},\mathcal N_{\underline\lambda})^G=H^1(G_{v_{\underline\lambda}},V/\mathfrak g.v_{\underline\lambda}).
$$
We have similarly
$$
H^1(G.v_{\underline\lambda},\mathcal O_{X_0}^{\oplus n})^G=
\oplus_{1\leq i,j\leq s} H^1(G_{v_{\underline\lambda}},K_{i,j}).
$$

From~\cite{Ho}, we know that
$$
\mathop{H^1}(G_{v_{\underline\lambda}}, V/\mathfrak g.v_{\underline\lambda})\simeq 
\mathop{H^1}(\mathfrak g_{v_{\underline\lambda}}, V/\mathfrak g.v_{\underline\lambda})^{G_{v_{\underline\lambda}}/G_{v_{\underline\lambda}}^\circ}.
$$
The proposition follows from the exact sequence~\ref{1srcharacterisationof T^2}.
\end{proof}

Let $S(\Gamma)$ be the set of simple roots of $G$ orthogonal to every element of $\Gamma$
and let $\Sigma(\Gamma)$ be the set of $T_\ad$-weights of the tangent space $T_{X_0}\Hilb^G_\Gamma$.
Take $\gamma\in\Sigma(\Gamma)$, we shall denote by $v_\gamma\in  \oplus_i V(\lambda_i)_{\lambda_i-\gamma}$ the corresponding weight vector.

\begin{proposition}\label{characterisationofH^1}
The $T_\ad$-weight vectors of $\mathop{H^1}(\mathfrak g_{v_{\underline\lambda}},V/\mathfrak g.v_{\underline\lambda})$ 
can be represented by the following  cocycles indexed by the simple roots $\alpha$ in $S\setminus S(\Gamma)$ and by the $T_\ad$-weights $\gamma$ in $\Sigma(\Gamma)$
\begin{eqnarray*}
\varphi_{\alpha,\gamma}:
 X_\alpha &\longmapsto& X_{-\alpha}^r v_{\gamma}  \\
 X_\delta &\longmapsto& 0  \quad\quad\mbox{if $\delta\neq \alpha$}  
\end{eqnarray*}
Here $r=-(\gamma,\alpha^\vee)$ if $\alpha\not\in\supp\gamma$ and $r=0$ otherwise.	
\end{proposition}

The proof of Proposition~\ref{characterisationofH^1} is postponed to Section~\ref{Proofs}.

\section{Wonderful varieties}\label{wonderful}

\begin{definition}
An algebraic $G$-variety $X$ is said to be \textit{wonderful of rank $r$} if it satisfies the following conditions.
\smallbreak
\noindent
{\rm(i)}\enspace
$X$ is smooth and complete.
\smallbreak
\noindent
{\rm(ii)}\enspace
$X$ contains an open $G$-orbit whose complement is the union of $r$ smooth prime $G$-divisors $D_1,\ldots,D_r$ with normal crossings and such that $\cap_1^r D_i\neq\emptyset$.
\smallbreak
\noindent
{\rm(iii)}\enspace
The $G$-orbit closures of $X$ are given by the intersections $\cap_I D_i$ where $I$ is a subset of $\{1,\ldots,r\}$.
\end{definition}

As examples of wonderful varieties, one may consider flag varieties or 
De Concini-Procesi compactifications of symmetric spaces; \textsl{see}~\cite{DP83}.

Wonderful varieties are projective and spherical; \emph{see}~\cite{L96}.

\begin{definition}
A wonderful variety whose points have a selfnormalising stabilizer is called \textsl{strict}.
\end{definition}

Luna introduced several invariants attached to any wonderful $G$-variety $X$: \textsl{spherical roots, colors}.
We shall recall freely results concerning these notions; \textsl{see}~\cite{L97,L01} for details. 

Let $Y$ be the (unique) closed $G$-orbit of $X$ and
$z\in Y$ be the unique point fixed by the Borel subgroup $B^-$ of $G$ such that $B\cap B^-=T$.

The \emph{spherical roots} of $X$ are the $T$-weights of the quotient
$T_z X/T_z Y$ where $T_z X$ (resp. $T_z Y$) denotes the tangent space at $z$ of $X$ (resp. of $Y$).
The rank of $X$ is equal to the number of spherical roots of $X$. 

Let $P_X$ be the stabilizer of the point $z\in Y$.
The subgroup $P_X$ is a parabolic subgroup hence it corresponds to a subset $S^p_X$ of the set of simple roots $S$.

In case of strict wonderful varieties, the couple $(S^p_X, \Sigma_X,\emptyset)$ shares nice properties: it is \textit{a spherical system for $G$}.
Spherical systems were introduced by Luna as triples which fulfill certain axiomatic conditions.

Luna's conjecture asserts that there corresponds a unique (non-necessarily strict) wonderful $G$-variety to a given spherical system.

Before recalling Luna's definition of spherical systems in case the third datum is the empty set, let us set some further notation.
Let $S$ be the set of simple roots of $G$ relatively to $B$ and $T$.
Given a simple root $\alpha\in S$, let $\alpha^\vee$ be its associated coroot, that is $\alpha^\vee=2\alpha/(\alpha,\alpha)$.
Given $\beta=\sum_\alpha n_\alpha \alpha$ where the sum runs over $S$ with $n_\alpha\geq 0$ (resp. $n_\alpha\leq 0$) for all $\alpha$.
The \textsl{support} of $\beta$ is defined as usual as the set of simple roots $\alpha$ such that $n_\alpha\neq 0$; we shall denote it by $\supp\beta$. 

\subsection{Spherical systems}

\begin{definition}
\textsl{The set of spherical roots of $G$} is the set of characters of Table~\ref{sphericalroots} whose support is a subset of $S$. 
We denote it by $\Sigma(G)$.
\end{definition}

\begin{definition}(\cite[1.1.6]{BL08})
Let $S^p\subset S$ and $\sigma\in\sigma(G)$.
The couple $(S^p,\sigma)$ is said to be \textsl{compatible} if 
$$
S^{pp}(\sigma)\subset S^p\subset S^p(\sigma)
$$
where $S^{pp}(\sigma)$ is one of the following sets
\smallbreak\noindent
{\rm -}\enspace $S^p(\sigma)\cap\supp\sigma\setminus\{\alpha_r\}$ if $\sigma=\alpha_1+\ldots+\alpha_r$ with $\supp\sigma$ of type $\mathsf B_r$,
\smallbreak\noindent
{\rm -}\enspace $S^p(\sigma)\cap\supp(\sigma)\setminus\{\alpha_1\}$ if $\supp\sigma$ is of type $\mathsf C_r$,
\smallbreak\noindent
{\rm -}\enspace $S^p(\sigma)\cap\supp(\sigma)$ otherwise.

\end{definition}

\begin{definition}
Let $S^p\subset S$ and  $\Sigma\subset\Sigma(G)$.
The couple $(S^p,\Sigma)$ is called a \textsl{spherical system} if
\smallbreak
\noindent
{\rm($\Sigma 1$)}\enspace $(\alpha^\vee,\sigma)\in2\mathbb Z_{\leq 0}$ for all $\sigma\in\Sigma\setminus\{2\alpha\}$ and all $\alpha\in S $ such that $2\alpha\in \Sigma$.
\smallbreak
\noindent
{\rm($\Sigma 2$)}\enspace $(\alpha^\vee,\sigma)=(\beta^\vee,\sigma)$ for all $\sigma\in\Sigma$ and all $\alpha,\beta\in S$ which are mutually orthogonal and such that $\alpha+\beta\in\Sigma$.
\smallbreak\noindent
{\rm($S$)}\enspace The couple $(\{\sigma\},S^p)$ is compatible for any $\sigma\in\Sigma$.
\smallbreak\noindent
\smallbreak\noindent
{\rm($St$)}\enspace The couple $(\{2\sigma\},S^p)$ is not compatible for any $\sigma\in\Sigma$.
\end{definition}

\subsection{Colors}
Given a wonderful $G$-variety $X$, the $B$-stable but not $G$-stable prime divisors in $X$ are called the \textsl{colors of $X$}.
The subgroup $P_X$ (see the previous paragraph) coincides with the stabilizer of the colors of $X$.

\begin{definition}
\textsl{The set of colors $\Delta$} of a given spherical system $(S^p, \Sigma)$ is defined as the set of the following dominant weights:
\smallbreak
\noindent{\rm -}\enspace $\omega_\alpha$ (resp.  $2\omega_\alpha$) 
if $\alpha\in S\setminus S^p$ and $2\alpha\not\in\Sigma$ (resp. $2\alpha\in\Sigma$);
\smallbreak
\noindent{\rm -}\enspace $\omega_\alpha+\omega_\beta$ if $\alpha,\beta\in S$ and $\alpha+\beta\in\Sigma$.

Here $\omega_\alpha$ stands for the fundamental weight associated to any simple root $\alpha$, that is
$\omega_\alpha(\beta^\vee)$ equals $1$ if $\beta=\alpha$ and equals $0$ for any other simple root $\beta$.
\end{definition}

In case the spherical system is given by a wonderful $G$-variety $X$, the set $\Delta$ coincides with
the set of colors of $X$ in the following way.
Let $D$ be a color of $X$.
Consider its inverse image $\pi^{-1}(D)$ through the canonical quotient $\pi:G\rightarrow G/H$ 
where $G/H$ is isomorphic to the dense $G$-orbit of $X$.
Choose $H$  such that $BH$ is open in $G$ (recall that $X$ is $G$-spherical).
Then, whenever the Picard group of $G$ is trivial, $\pi^{-1}(D)$ is a $B\times H$-stable divisor of $G$ hence it has an equation $f_D$ which is a $B\times H$-eigenvector.
The set $\Delta$ is thus given by the set of the $B$-weights of the $f_D$'s. 

\section{Invariant Hilbert schemes and Wonderful varieties}\label{connection}

\subsection{Tangent spaces and spherical systems}

Given a saturated monoid $\Gamma$, recall the definition of the sets $\Sigma(\Gamma)$ and $S^p(\Gamma)$; \textsl{see} the paragraph right before Proposition~\ref{characterisationofH^1}.

\begin{theorem}(\cite[Theorem 4.1]{BC1})
The couple $(S^p(\Gamma),\Sigma(\Gamma))$ is a  spherical system.
\end{theorem}

Denote by $\Delta(\Gamma)$ the set of colors of $(S^p(\Gamma),\Sigma(\Gamma))$.
Let $S(\gamma)$ be the set of simple roots $\delta$ such that $\gamma-\delta$ is a root.

\begin{proposition}\label{non-rigidity}
Assume that the tangent space $T_{X_0}\Hilb^G_\Gamma$ is not trivial. 
Let $\lambda$ be one of the dominant weights defining the monoid $\Gamma$.
If the support of the set $\Sigma(\Gamma)$ coincides with the set of simple roots $S$ then one of the following possibilities may occur.
\smallbreak{\rm -}\enspace
$\lambda$ belongs to $\Delta(\Gamma)$ up to a scalar,
\smallbreak{\rm -}\enspace
$\lambda$ equals $\omega_\alpha+\omega_\beta$ if there exists $\gamma\in\Sigma(\Gamma)$ such that 
$(\gamma,\alpha)>0$ and $(\gamma,\beta)>0$,
\smallbreak{\rm -}\enspace
$\lambda=\omega_\alpha+\sum a_\delta\omega_\delta$ with $(\gamma,\delta)=0$ along with $S(\gamma)\neq\{\alpha,\beta\}$ for all $\gamma\in\Sigma(\Gamma)$.

In case $\lambda=a\lambda_D$ with $a>1$ and $\lambda_D\in\Delta(\Gamma)$, the set $\Sigma(\Gamma)$ has to be a singleton.
Further, if the second possibility occurs for some $\lambda$ then $(\lambda',\alpha+\beta)=0$ for all $\lambda'\neq\lambda$.
\end{proposition}

As a converse, we have

\begin{proposition}\label{converse}
Given a spherical system $(S^p,\Sigma)$ of $G$, let $\Delta$ be its set of colors.
Let $\Gamma$ be a saturated monoid spanned by linearly independent dominant weights which satisfy the conditions stated in the previous proposition (with $\Delta(\Gamma)=\Delta$ and $\Sigma(\Gamma)=\Sigma$).
Then the tangent space of the corresponding invariant Hilbert scheme is not trivial.
Further the elements of $\Sigma$ are $T_\ad$-weights of this tangent space.
\end{proposition}

\begin{remark}
In Section 6 of \cite{BC1}, one can find analogous versions of the two above statements whose proofs are quite indirect since requiring the use of wonderful varieties. We will provide a purely combinatorial proof; \textsl{see} Section~\ref{Proofs}.
\end{remark}

\begin{definition}
A spherical system $(S^p,\Sigma)$ of $G$ is \textsl{primitive} if
the support of $\Sigma$ coincides with the whole set $S$ of simple roots of $G$,
and if there is no spherical system $(S^p,\Sigma')$ such that $\Sigma'$ strictly contains $\Sigma$.
\end{definition}

\subsection{Classification of wonderful varieties}

Given the decomposition of $V$ into irreducible $G$-modules,
$$
V= V(\lambda_1)\oplus\ldots\oplus V(\lambda_s)
$$ 
the corresponding $s$-multi-cone $\mathcal C(X)$ generated by any $X\subset V$ is (after~\cite{BK})
$$
\cup_{\underline t\in\mathbb C^s}\underline t.X
$$
where
$$
\underline t.X=\{(t_1 x_1,\ldots,t_s x_s): \underline t=(t_i)_i\in\mathbb C^s, (x_i)_i\in X\}.
$$

\begin{theorem}\label{classification}
Take a primitive spherical system $(S^p,\Sigma)$ and let $\Gamma$ be the monoid spanned by its set of colors $\lambda_i$ ($i=1,\ldots,s$).
Consider the invariant Hilbert scheme $\Hilb_{\underline\lambda}^G$.
Let $X_1$ be a closed point such that its $T_\ad$-orbit is dense within $\Hilb_{\underline\lambda}^G$.
Regarding $X_1$ as a subvariety of $V=\oplus_i V(\lambda_i)$, let $\mathcal C(X_1)$ be the corresponding $s$-multi-cone.

Then the multihomogeneous spectrum of the regular ring $\mathcal R(\mathcal C(X_1))$ of $\mathcal C(X_1)$
$$
X_\Gamma=\mathrm{Proj}\mathcal R(\mathcal C(X_1))
$$
is a wonderful $G$-variety with spherical system $(S^p,\Sigma)$.
\end{theorem}

\begin{proof}
Recall that by Proposition~\ref{converse} and Theorem~\ref{smoothness}, 
the invariant Hilbert scheme $\Hilb_{\underline\lambda}^G$ is not trivial whence the existence of $X_1$ by the proof of Corollary~\ref{affinity}.
Further the dimension of $\Hilb_{\underline\lambda}^G$ equals the cardinality of $\Sigma$.

Let $v\in V$ be such that $X_1$ is the $G$-orbit closure of $v$ within $V$.
Since the dominant weights $\lambda_i$ defining the monoid $\Gamma$ are linearly independent, 
the $G\times T_\ad$-orbit closure of $v$ within $V$ coincides with the multicone $\mathcal C(X_1)$ generated by $X_1$.

Write $v=v_1+\ldots+v_s$ with $v_i\in V(\lambda_i)$.
Since the saturation assumption is fulfilled by the set of colors $\lambda_i$, 
the $G$-variety $X_\Gamma$ coincides with the $G$-orbit closure of
$([v_1],\ldots,[v_s])$ within the multiprojective space $\mathbb P\left(V(\lambda_1)\right)\times\ldots\times\mathbb P\left(V(\lambda_s)\right)$.

The projective variety $X_\Gamma$ is thus smooth and spherical for the action of $G$ with a single closed $G$-orbit given by the multihomogeneous spectrum of the regular ring of $\overline{G.v_{\underline\lambda}}\subset V$.
Being also toroidal, the variety $X_\Gamma$ is thus a wonderful $G$-variety (\emph{see}~\cite{L97}).
Further its rank equals the cardinality of the given set $\Sigma$.
We shall prove that the spherical system $(S^p_X,\Sigma_X)$ of $X_\Gamma$ is indeed $(S^p,\Sigma)$.
This follows readily from the characterization of the closed $G$-orbit of $X_\Gamma$ and 
from Corollary~\ref{affinity} along with the definition of spherical roots recalled in the previous section.
\end{proof}

As already noticed and proved in~\cite{BC1} (\textsl{see} Corollary 2.5 in [loc.cit.]; \textsl{see} also \cite{Br2}),
we get the following. 

\begin{corollary}
Given a connected reductive algebraic group $G$ and $\Gamma$ a saturated monoid,
consider the set $\tilde\Gamma$ of colors of the spherical system $(S^p(\Gamma),\Sigma(\Gamma))$.
Let $X_{\tilde\Gamma}$ be the wonderful $G$-variety obtained in the above theorem.
The universal family of $\Hilb_{\underline\lambda}^G$ ($\underline\lambda$ defining the monoid $\Gamma$)
is given by the quotient map 
$$
\tilde X_{\tilde\Gamma}
\longrightarrow \tilde X_{\tilde\Gamma}//G
$$
where $\tilde X_{\tilde\Gamma}$ denotes the normalization of the affine multi-cone  
$$
\mathrm{Spec}\oplus_{\underline\nu\in\Gamma}H^0(X_{\tilde\Gamma},\mathcal L_{\underline\nu}).
$$
\end{corollary}

Here $\mathcal L_\chi$ denotes the $G$-linearized invertible sheaf associated to the character $\chi$.

As a consequence of the two above statements, we can answer positively to Luna's conjecture in the context of strict wonderful varieties.

\begin{corollary}
Given a spherical system $(S^p,\Sigma)$ of some reductive algebraic group $G$,
there exists a unique wonderful $G$-variety whose spherical system is $(S^p,\Sigma)$.
\end{corollary}

\begin{proof}
By~\cite{L01}, it suffices to consider primitive spherical systems.
The existence part thus follows from Theorem~\ref{affinity}; the uniqueness is a consequence of the corollary above as already
noticed in Section 6.2 of~\cite{BC2}.
\end{proof}

\begin{remark}
The existence part  of the above theorem is proved in~\cite{BC2} by different methods; the proof follows Luna's approach introduced in ~\cite{L01}.
This proof is Lie theoretical.
More specifically, for a given spherical system of $G$, a subgroup of $G$ is exhibited. 
One thus proves that this subgroup is wonderful (it has a compactification which is a wonderful variety)
and by ad-hoc arguments that it has indeed the spherical system under consideration.
\textsl{See}~\cite{BP} and \cite{B} for other partial positive answers to Luna's conjecture.

The uniqueness part (in full generality) of this theorem follows from~\cite{Lo}. The approach developed here uses the so-called colored fans introduced in \cite{LV}.
\end{remark}

\section{Proofs}\label{Proofs}

\subsection{Auxiliary lemmas}

For convenience, we shall recall the following statements from~\cite{BC1}.

\begin{lemma}(\cite[Proposition 3.4]{BC1})\label{rootsupport}
Let $\gamma$ be a $T_\ad$-weight vector of $(V/\mathfrak g.v_{\underline{\lambda}})^{G_{v_{\underline{\lambda}}}}$.
If $\delta$ is a simple root in the support of $\gamma$ such that $\gamma-\delta$ is not a root then $(\gamma,\delta)\geq 0$.
Further if $\gamma$ and $\delta$ are orthogonal then $\delta$ is orthogonal to all the $\lambda_i$'s.
\end{lemma}

\begin{remark}
We will generalize the above statement in Proposition~\ref{supportofGamma}.
\end{remark}

\begin{lemma}(\cite[Proof of Theorem~3.10]{BC1})\label{weightvector}
If $[v]$ is a $T_\ad$-weight vector of $(V/\mathfrak g.v_{\underline{\lambda}})^{G_{v_{\underline{\lambda}}}}$,
then one of its representatives $v\in V$ can be taken as follows
\smallbreak
\noindent
$$[v]\in (V(\lambda)/\mathfrak g.v_{\lambda})^{G_{v_{{\lambda}}}}\quad\mbox{ or }\quad v=X_{-\gamma}v_\lambda 
$$
where $\lambda$ is one of the given dominant weights $\lambda_i$.
The second case occurs only when $(V(\lambda)/\mathfrak g.v_{\lambda})^{G_{v_{{\lambda}}}}$ is trivial.
\end{lemma}

\begin{lemma}(\cite[Lemma 3.13]{BC1})\label{f4}
If $G$ is of type $\mathsf F_4$, 
$\lambda_1=\omega_4+a\omega_3$ and $\lambda_i=a_i\omega_3$ for $a_i>0$ for some $i$, then the space 
$(V/\mathfrak g.v_{\underline{\lambda}})^{G_{v_{\underline{\lambda}}}}$
is trivial.
\end{lemma}

\begin{remark}
The above lemma is basically the same as Lemma 3.13 in~\cite{BC1}. Here we have not required that $a$ be non-zero.
The proof conducted there is however still valid.
\end{remark}

\bigbreak

\subsection{\textbf{Proof of Proposition~\ref{characterisationofH^1}}}

Let 
$$
V=V(\lambda_1)\oplus\ldots\oplus V(\lambda_s)
$$
where the dominant weights $\lambda_i$ satisfy the properties of Proposition~\ref{non-rigidity}.

Let $\dot{\varphi}$ be a $T_\ad$-weight vector of $H^1(\mathfrak g_{v_{\underline\lambda}}, V/\mathfrak g.v_{\underline\lambda})^{G_ {v_{\underline\lambda}}/G_ {v_{\underline\lambda}}^\circ}$.
We shall prove that $\varphi$ can be represented by some cocyle $\varphi=\varphi_{\alpha,\gamma}$ as follows.
Let $\alpha$ be a simple root of the isotropy Lie algebra $\mathfrak g_{v_{\underline\lambda}}$ and
$\gamma$ be either a $T_\ad$-weight in $\Sigma(\Gamma)$ or one of the given dominant weights $\lambda_i$.
If $\gamma\in\Sigma(\Gamma)$, let 
$[v_\gamma]\in \left(V/\mathfrak g.v_{\underline\lambda}\right)^{G_{v_{\underline\lambda}}}$ be the corresponding $T_\ad$-weight vector
otherwise take $v_\gamma$ to be the highest weight vector $\lambda_i$.
Then define
\begin{eqnarray*}
\varphi_{\alpha,\gamma}: X_\alpha  &\longmapsto&  [X_{-\alpha}^r v_{\gamma}]  \\
                         X_\delta  &\longmapsto&   0 													\quad\mbox{ if $\delta\neq \alpha, \delta\in S$}  
\end{eqnarray*}
where 
$$
r=\mathrm{max}\{i\geq 0: X_{-\alpha}^i v_\gamma\neq 0 \mbox{ in } V\}
\quad\mbox{ and }\quad 
v_\gamma\in V(\lambda_j)_{\lambda_j-\gamma}\mbox{ for some $\lambda_j$}.
$$

By Kostant's theorem (\textsl{see}~\cite{K61} and also \cite{K02}, Chap. III-2), whenever $\delta\in S^p(\Gamma)$, namely $\delta$ is orthogonal to every dominant weight $\lambda_i$ then $\dot{\varphi}(X_\alpha)=0$.
Hence in the definition of $\varphi_{\alpha,\gamma}$ above, the simple root $\alpha$ does not belong to $S^p(\Gamma)$.

\begin{lemma}
Let $[v_\gamma]\in \left(V/\mathfrak g.v_{\underline\lambda}\right)^{G_{v_{\underline\lambda}}}$ be a $T_\ad$-weight vector
and
take $v_\gamma\in V(\lambda_j)_{\lambda_j-\gamma}$ for some $\lambda_j$ to be one of its representatives.
Suppose that $\varphi_{\alpha,\gamma}$ defines a cocyle for some simple root $\alpha$.
Then the vector
$$
[X_{-\alpha}^r v_\gamma]\quad\mbox{ with }\quad r=\mathrm{max}\{i\geq 0: X_{-\alpha}^i v_\gamma\neq 0 \mbox{ in } V\}
$$
does not depend upon the choice of the dominant weight $\lambda_j$.
\end{lemma}

\begin{proof}
As recalled in Lemma~\ref{weightvector}, the choice of the dominant weight $\lambda_j$ is not unique whenever
$v_\gamma=X_{-\gamma}v_{\lambda_j}$.
More specifically, we can choose also the representative $v_\gamma$ to be $X_{-\gamma}v_{\lambda_i}$ for some other $\lambda_i\neq \lambda_j$.
If $\lambda_i$ and $\lambda_j$ are both orthogonal to the given simple root $\alpha$ then by Proposition~\ref{non-rigidity} then so is $\gamma$.
Further $\alpha$ can not belong to the support of $\gamma$; the lemma follows.
Suppose that $\lambda_i$ is not orthogonal to $\alpha$ and that $\alpha$ does not belong to the support of $\gamma$.
Then one checks that $\varphi_{\alpha,\gamma}$ can not be a cocycle.
We are thus left with $\lambda_i$ non-orthogonal to $\alpha$ and $\alpha$ lying in the support of $\gamma$.
By Table~\ref{sphericalroots} and Proposition~\ref{rootsupport}, $\gamma-\alpha$ is a root. 
In case $\gamma+\alpha$ is a root, the lemma is obvious otherwise one has that $(\lambda,\alpha^\vee)=1$ and $(\lambda_j,\alpha^\vee)=0$.
The lemma follows.
\end{proof}

\begin{lemma}
Let $\alpha$ be a simple root.
Suppose there exists a unique dominant weight $\lambda_i$ which is not orthogonal to $\alpha$.
Let $[v]\in V/\mathfrak g.v_{\underline\lambda}$ be non-trivial.
Then either there exists a positive root $\beta\neq \alpha$ such that $X_\beta.v\neq 0$ in $V$ or $v=X_{-\alpha}^r v_{\lambda_i}$ in $V$
with $r=2$ or $4$.
\end{lemma}

\begin{proof} 
If $X_\beta v=0$ for all positive roots $\beta$ distinct to $\alpha$
then $v$ will be equal either to $v_{\lambda_i}$, $X_{-\alpha} v_{\lambda_i}$ or to $X_{-\alpha}^r v_{\lambda_i}$ with $r>1$ 
for some $\lambda_i$.
Under the assumptions of the lemma, the class  of $v$ in $V/\mathfrak g.{v_{\underline \lambda}}$ will be trivial unless it is equal to the latter possibility. 
By Proposition~\ref{non-rigidity}, the integer $r$ equals $2$ or $4$.
\end{proof}

\begin{remark}
The condition of the previous lemma is fulfilled whenever the $\lambda_i$'s are the colors of a spherical system.
\end{remark}

\begin{lemma}\label{invariancecondition}
We have $X_\beta\varphi(X_\alpha)=0$ in $V/\mathfrak g.v_{\underline\lambda}$ for every root $\beta\neq \alpha$ of the isotropy Lie algebra
$\mathfrak g_{v_{\underline\lambda}}$.
\end{lemma}

\begin{proof}
Let $\mu$ be the $T_\ad$-weight of $\varphi(X_\alpha)$ then $\varphi(X_\beta)$ should be of $T_\ad$-weight $\mu-\alpha+\beta$, for any root $\beta\neq\alpha$ of the isotropy Lie algebra $\mathfrak g_{v_{\underline\lambda}}$.

Let $\beta$ be a simple root. Suppose that $X_\beta\varphi(X_\alpha)\not\in\mathfrak g.v_{\underline\lambda}$.
Since $X_\beta\varphi(X_\alpha)-X_\alpha\varphi(X_\beta)\in\mathfrak g.v_{\underline\lambda}$,
it follows that $X_\alpha\varphi(X_\beta)\not\in\mathfrak g.v_{\underline\lambda}$.
And in particular $\varphi(X_\beta)$ is then a weight vector of weight $\mu-\beta+\alpha$, which is possible only in case $\beta=\alpha$.
The lemma follows.
\end{proof}

The following proposition is the announced generalization of Lemma~\ref{rootsupport}.

\begin{proposition}\label{supportofGamma}
Take a $T_\ad$-weightvector $[v_\gamma]\in V/\mathfrak g.v_{\underline\lambda}$ of weight $\gamma$.
Let $\alpha,\delta$ be orthogonal simple roots with $\delta$ in the  support of $\gamma$.
Suppose that $\gamma-\delta$ is not a root and that
$$
[X_\beta v_\gamma]=0\quad\mbox{ for all $\beta\neq \alpha$}.
$$
Then every $\lambda_i$ is orthogonal to $\delta$.
\end{proposition}

\begin{proof}
We start similarly as for the proof of Proposition 3.4 in~\cite{BC1}.
Suppose there exists a dominant weight among the given $\lambda_i$'s, say $\lambda_k$ which is not orthogonal to $\delta$ then in particular such a $\lambda_k$ is not orthogonal to $\gamma$ neither since $\delta$ lies in the newort of $\gamma$.
 
Note that if $\gamma$ is not equal to $\alpha$ (up to a scalar) then there exists
a positive root $\beta\neq \alpha$ whose support is contained in that of $\gamma$
and such that 
$$
(\lambda_k-\gamma,\beta)< 0.
$$

We claim that the component $v_\gamma^k$ of $v_\gamma$ in $V(\lambda_k)$ can be assumed to be trivial.
Indeed if $v_\gamma^k\neq 0$ then $X_\beta v_\gamma\neq 0$ in $V$ because of the inequality stated right above.
And in turn, we will get that $X_\beta v_\gamma=X_{-\gamma+\beta}v_{\underline\lambda}$ by Lemma~\ref{invariancecondition}.
Considering instead the representative $v_\gamma-X_{-\gamma+\beta}v_{\underline\lambda}$ whose component in $V(\lambda_k)$ is trivial, we shall obtain the claim.

Assume thus for the remaining of the proof that $v_\gamma^k=0$.
It follows that if $\lambda_k$ is not orthogonal to $\delta$ then $\delta$ should lie in the support of $\beta'$
where $\beta'$ is a positive root distinct to $\alpha$ such that 
$$
0\neq X_{\beta'}v_\gamma\in\mathfrak g.v_{\underline\lambda}.
$$
If such a root $\beta'$ can be taken to be simple then $\beta'=\delta$ whence a contradiction since $\gamma-\delta$ is not a root.

We are thus left with the situation where $X_\delta v_\gamma=0$ in $V$ for all  simple roots $\delta\neq\alpha$.
Since $X_{\beta'}v_\gamma\neq 0$ in $V$ for some positive root $\beta'$ as above, it follows that
$X_{\alpha'+\alpha}\neq 0$ in $V$ whenever $\alpha'$ is a simple root in the support of $\gamma$ and adjacent to $\alpha$.
Therefore $\delta$ has to be adjacent to $\alpha$.
\end{proof}

\begin{proof} [Proof of Proposition~\ref{characterisationofH^1}]

Let $\gamma$ be the $T_\ad$-weight of $[\varphi(X_\alpha)]$ and denote by $v_\gamma$ a representative of $\varphi(X_\alpha)$ in $V/\mathfrak g.v_{\underline\lambda}$.
Note first that whenever $\alpha$ does not belong to the support then $X_\alpha v_\gamma=0$ in $V$ 
hence $[v_\gamma]\in\left(V/\mathfrak g.{v_{\underline\lambda}}\right)^{G_{v_{\underline\lambda}}}$ by Lemma~\ref{invariancecondition}. 

We shall assume in the remaining of the proof that $\alpha$ does belong to the support of the $T_\ad$-weight $\gamma$.
We shall proceed along the type of the support of $\gamma$.
Let us work out a few cases in details.
The main ingredients are Proposition~\ref{supportofGamma} and Proposition~\ref{non-rigidity}.
Let $v_\gamma$ be non-equal to $X_{-\alpha}^r v_{\lambda_i}$ otherwise the proposition is already proved.
As a consequence of Lemma~\ref{invariancecondition}, the weight $\gamma$ can be written as a sum of two positive roots, say $\beta_1$ and $\beta_2$.

Consider first the case where the supports of the roots $\beta_1$ and $\beta_2$ are orthogonal.
Then by Proposition~\ref{supportofGamma} and its proof, the roots $\beta_1$ and $\beta_2$ have to be simple.
Applying now Proposition~\ref{non-rigidity}, we obtain that there is a single dominant weight, say $\lambda$,
which is neither orthogonal to $\beta_1$ nor to $\beta_2$.
Further if this dominant weight $\lambda$ is a color then $\gamma\in\Sigma(\Gamma)$ otherwise there may exist a second dominant weight non-orthogonal
to one of the roots $\beta_i$, say $\beta_1$. The latter possibility is ruled out.
First note that in that case, $G_ {v_{\underline\lambda}}$ is not connected (\textsl{see}~Proposition~\ref{non-rigidity}).
Further $\alpha$ has to equal $\beta_2$ and $\varphi$ is of weight $\beta_1$.
But $\beta_1$ can not be written as integral combination of the $\lambda_i$'s (\textsl{see}~Proposition~\ref{non-rigidity}) hence $\varphi$ is not fixed by $G_ {v_{\underline\lambda}}/G_ {v_{\underline\lambda}}^\circ$.
 
Suppose now that the support of $\gamma$ is of type $\mathsf A_n$.
From Proposition~\ref{supportofGamma}, we obtain that if $\gamma$ is not a root then $\gamma=\alpha_{i-1}+2\alpha_i+\alpha_{i+1}$ with $\alpha=\alpha_i$.
Thanks to Proposition~\ref{non-rigidity}, all the dominant weights $\lambda_k$ have to be orthogonal to both $\alpha_{i-1}$ and $\alpha_{i+1}$. 
One thus has clearly that
$[v_\gamma]\in\left(V/\mathfrak g.{v_{\underline\lambda}}\right)^{G_{v_{\underline\lambda}}}$ by Lemma~\ref{invariancecondition}.
If $\gamma$ is now a root then one gets: $\gamma=\alpha_i+\ldots+\alpha_j$.
Then Proposition~\ref{supportofGamma} (and its proof) along with Lemma~\ref{invariancecondition} yield: $\alpha=\alpha_i$ or $\alpha=\alpha_j$.
Therefore if the dominant weights are orthogonal to all the simple roots in the support of $\gamma$ except $\alpha_i$ and $\alpha_j$ 
then $[v_\gamma]\in\left(V/\mathfrak g.{v_{\underline\lambda}}\right)^{G_{v_{\underline\lambda}}}$. 
If the dominant weights are not so then obviously $\gamma-\alpha\in\Sigma(\Gamma)$.
Note that the same arguments  can be applied to $\gamma=\alpha_i+\ldots+\alpha_j$ with support of type $\mathsf B_n$.

In case of type $\mathsf B_n$, the weight vector $v_\gamma$ may be equal to $X_{-\alpha}^r.v_\lambda$ whenever $\alpha=\alpha_n$, $\alpha_1+\ldots+\alpha_n\in\Sigma(\Gamma)$ and there are at least two dominant weights among the given $\lambda_i$'s which are not orthogonal to $\alpha_n$. Suppose $\gamma$ is a root.
Similarly as before, one obtains that $\gamma=\alpha_i+\ldots+\alpha_n$ or $\gamma=2(\alpha_i+\ldots+\alpha_n)$ if $\alpha$ lies in the support of $\gamma$. Let $\gamma=2(\alpha_i+\ldots+\alpha_n)$. Then since $\gamma-\delta$ is not a root for each simple root distinct to $\alpha_i$, 
$X_{\alpha_i}.v_\gamma\neq 0$ in $V$ and Proposition~\ref{supportofGamma} gives: $\alpha=\alpha_i$ otherwise $\alpha$ will be in $S^p$- a contradiction. Finally, one gets that $[v_\gamma]\in\left(V/\mathfrak g.{v_{\underline\lambda}}\right)^{G_{v_{\underline\lambda}}}$ whenever
$\alpha_{i+1}\in S^p$ otherwise $\gamma-2\alpha\in\Sigma(\Gamma)$.

The other types can be conducted similarly.

\end{proof}

\subsection{\textbf{Proof of Theorem~\ref{smoothness}}}

To show the smoothness of the invariant Hilbert scheme (Theorem~\ref{smoothness}),
we shall use the characterization of the second cotangent module given in Proposition~\ref{characterisation}
and thus prove the following statement by means of Proposition~\ref{characterisationofH^1}.

Let 
$$
V=V(\lambda_1)\oplus\ldots\oplus V(\lambda_s)
$$
where the dominant weights $\lambda_i$ satisfy the properties of Proposition~\ref{non-rigidity}.

For short, set
$$
S^2 V/V(2\underline\lambda)=\oplus_{1\leq i,j\leq s} V(\lambda_i)\cdot V(\lambda_j)/V(\lambda_i+\lambda_j).
$$ 

Take $\dot{\varphi}$ in $\mathop{H^1}(\mathfrak g_{v_{\underline\lambda}},V/\mathfrak g.v_{\underline\lambda})$
and denote by $\varphi$ a cocycle representing $\dot{\varphi}$.
Recall the definition of the map $f$: let $X_\alpha$ be any root vector of the isotropy Lie algebra $\mathfrak g_{v_{\underline\lambda}}$, we have
$$
f{\varphi}(X_\alpha)=\sum_i\varphi(X_\alpha)\cdot v_{\lambda_i}.
$$

\begin{proposition}\label{triviality}
The map 
$$
f:H^1(\mathfrak g_{v_{\underline\lambda}},V/\mathfrak g.v_{\underline\lambda})\rightarrow H^1(\mathfrak g_{v_{\underline\lambda}},S^2 V/V(2\underline\lambda))
$$
is injective.
\end{proposition}

\begin{proof}
To prove the injectivity of $f$, we can assume without loss of generality that $\varphi$ is a $T_\ad$-weight vector.
By Proposition~\ref{characterisationofH^1}, $\varphi=\varphi_{\alpha,\gamma}$ for some $\alpha$ simple and $\gamma\in S(\Gamma)$ or $\gamma=\lambda_i$.
Set for convenience $\varphi(X_\alpha)=[v_{s_\alpha*\gamma}]$.

It is thus enough to prove that there exists  $v_{s_\alpha*\gamma}\cdot v_{\lambda_i}$ non-trivial in $\mathop S^2V/V(2\underline\lambda)$
for which
there is no $v\in\mathop S^2V/V(2\underline\lambda)$ such that $v_{s_\alpha*\gamma}\cdot v_{\lambda_i}=X_\alpha v$ in $\mathop{S^2}V/V(2\underline\lambda)$.

Note that this assertion holds whenever 
\begin{equation}\label{1tobeproved}
X_{-\alpha}\bigl(v_{s_\alpha*\gamma}\cdot v_{\lambda_i}\bigr)=0\quad\mbox{ in }\mathop{S^2}V/V(2\underline\lambda)
\end{equation} 
or
\begin{equation}~\label{2tobeproved}
X_\alpha^a v_{s_\alpha*\gamma}\neq 0\mbox{ in $V$}\quad\mbox{ for } a=(\lambda_i,\alpha^\vee).
\end{equation}

Let us consider first $X_{-\alpha}\bigl(v_{s_\alpha*\gamma}\cdot v_{\lambda_i}\bigr)$.
Note that by definition of $r$, we have: $X_{-\alpha}v_{s_\alpha*\gamma}=0$ in $V$.
We thus have 
$$
X_{-\alpha}\bigl(v_{s_\alpha*\gamma}\cdot v_{\lambda_i}\bigr)=v_{s_\alpha*\gamma}\cdot X_{-\alpha}v_{\lambda_i}.
$$

When $\lambda_i$ is orthogonal to $\alpha$, assertion~(\ref{1tobeproved}) to be proved is thus clear.

Suppose that $\lambda_i$ is not orthogonal to $\alpha$.
If $(\lambda_i,\gamma)\neq 0$, we have $(\gamma,\alpha)\geq 0$ by Proposition~\ref{non-rigidity}.
We shall prove assertion~(\ref{2tobeproved}) considering the cases where $(\lambda_i,\gamma)=0$ and $(\lambda_i,\gamma)\neq 0$ separately; this is done in the next lemmas.
\end{proof}

\begin{lemma}\label{lemma1injectivity}
Let $\varphi(X_\alpha)=X_{-\alpha}^r v_{\lambda_j}$ for some $\lambda_j$.
Then Assertion~(\ref{2tobeproved}) holds with $\lambda_i=\lambda_j$.
\end{lemma}

\begin{proof}
Note that $r=(\lambda_j,\alpha^\vee)$.
Hence if $r>1$ then $X_{-\alpha}^r v_{\lambda_j}. v_{\lambda_j}\neq 0$ and Assertion~(\ref{2tobeproved}) is clear.
If $r=1$ then there exists $\lambda_i\neq \lambda_j$ which is not orthogonal to $\alpha$.
But by Proposition~\ref{non-rigidity} this implies that such a $\lambda_i$ can be chosen to be fundamental.
Assertion~(\ref{2tobeproved}) follows with that chosen $\lambda_i$. 
\end{proof}

\begin{lemma}\label{lemma2injectivity}
Let $v_{s_\alpha*\gamma}\cdot v_{\lambda_i}\neq 0$ in $\mathop S^2V/V(2\underline\lambda)$ with $\gamma\in \Sigma(\Gamma)$.
If $\lambda_i$ is orthogonal to $\gamma$  then Assertion~(\ref{2tobeproved}) holds.
\end{lemma}

\begin{proof}
Note first that the support of $\gamma$ does not contain $\alpha$.
Indeed $\lambda_i$ being orthogonal to $\gamma$ it can not be orthogonal to $\alpha$ otherwise $v_{s_\alpha*\gamma}\cdot v_{\lambda_i}$ will be $0$.
Hence $(\gamma,\alpha^\vee)\leq 0$ and further $X_\alpha v_\gamma=0$ in $V$.
It follows in turn that $v_{s_\alpha*\gamma}=X_{-\alpha}^r v_\gamma$ with $r=(\lambda-\gamma,\alpha^\vee)$ and $v_\gamma\in V(\lambda)$.
Further, since $v_{s_\alpha*\gamma}.v_{\lambda_i}\neq 0$ and $(\lambda_i,\gamma)=0$, we have:
$v_{s_\alpha*\gamma}\neq v_\gamma$.

The weight $\lambda$ being non-orthogonal to $\gamma$, it is different from $\lambda_i$.

Assume that $(\gamma,\alpha)=0$ then since $\alpha$ does not belong to the support of $\gamma$, 
it has to be orthogonal to every simple root lying in the support of $\gamma$.
Let $\delta$ be a simple root $\delta$ in the support of $\gamma$ such that $X_\delta v_\gamma\neq 0$ in $V$.
Considering $X_\delta\varphi(X_\alpha)$ along with $(\lambda,\alpha)$, one ends up with a contradiction.

We deduce that $(\gamma,\alpha)<0$.
Thanks to Proposition~\ref{non-rigidity}, $\lambda_i$ is the single weight among the $\lambda_j$'s which is not orthogonal to $\alpha$.
Recall that $\gamma$ belongs to $\mathbb Z\Gamma$.
It follows that $r=(\lambda-\gamma,\alpha^\vee)\geq (\lambda_i,\alpha^\vee)$ whence the lemma.
\end{proof}

\begin{lemma}
Suppose $(\gamma,\alpha^\vee)>1$ with $\gamma\in \Sigma(\Gamma)$.
Then Assertion~(\ref{2tobeproved}) holds with $\lambda_i$ such that $v_\gamma\in V(\lambda_i)$.
\end{lemma}

\begin{proof}
If $\gamma=2\alpha$, we fall in the case of Lemma~\ref{lemma1injectivity}.

By Table~\ref{sphericalroots}, the weight $\gamma$ under consideration is such that $(\gamma,\alpha^\vee)=2$.
Further $\gamma-\delta$ is a root with $\delta$ a simple root if and only if $\delta$ equals $\alpha$. 
Hence $X_\delta v_\gamma=0$ in $V$ for all simple roots $\delta$ distinct to $\alpha$ and a fortiori $X_\alpha v_\gamma\neq 0$ in $V$.
Moreover $\gamma-2\alpha$ not being a root, we have: $X_{\alpha}^2 v_\gamma=0$ in $V$.
Since $(\lambda_i,\alpha^\vee)$ equals $1$ or $2$, we have respectively $v_{s_\alpha*\gamma}$ equals $v_\gamma$ or $X_{-\alpha} v_\gamma$.
The lemma follows readily. 
\end{proof}

\begin{lemma}
Let $\gamma\in \Sigma(\Gamma)$ and $(\gamma,\alpha^\vee)=1$ for some simple root $\alpha$.
Then Assertion~(\ref{2tobeproved}) holds.
\end{lemma}

\begin{proof}
Note that $\gamma-\alpha$ is a root; \textsl{see} Table~\ref{sphericalroots}.
By Proposition~\ref{non-rigidity}, there exists a unique $\lambda$ non-orthogonal to $\alpha$ and $(\lambda,\alpha^\vee)=1$.
Further $v_\gamma$ can be chosen in $V(\lambda)$.
If $X_\alpha v_\gamma=0$ in $V$, it follows from ~\cite{Js} and Lemma~\ref{weightvector} that $[v_\gamma]=[X_{-\gamma} v_{\lambda}]=[X_{-\gamma} v_{\lambda_j}]$
for some $\lambda_j\neq \lambda$ and such that $(\lambda_j,\gamma-\alpha)\neq 0$.
In particular $X_\alpha v_\gamma\neq 0$ in $V$ for $v_\gamma=X_{-\gamma}v_{\lambda_j}$. 
Then $v_\gamma$ can be chosen such that $v_\gamma\in V(\lambda_k)$ with  $X_\alpha v_\gamma\neq 0$ in $V$ and $\lambda_k=\lambda$ or $\lambda_j$ as above.
Assertion~(\ref{2tobeproved}) thus holds with $\lambda_i=\lambda$.
\end {proof}

\begin{lemma}
Let $\alpha$ be a simple root non-orthogonal to the monoid $\Gamma$ ($\alpha\not\in S^p(\Gamma)$).
Suppose $(\gamma,\alpha)=0$ then Assertion~(\ref{1tobeproved}) or Assertion~(\ref{2tobeproved}) holds.
\end{lemma}

\begin{proof}
First assume that $\alpha$ does not belong to the support of $\gamma$.
Then $\alpha$ is orthogonal to every simple root in the support of $\gamma$.
Let $v_\gamma\in V(\lambda)$. It follows that $v_{s_\alpha*\gamma}=v_\gamma$ in $V$ if and only if $(\lambda,\alpha)=0$.
If $v_\gamma\cdot v_\lambda\neq 0$ then Assertion~(\ref{1tobeproved}) holds whenever $(\lambda,\alpha)=0$.
If $v_\gamma\cdot v_\lambda=0$ then $v_\gamma=X_{-\gamma}v_\lambda$ and there exists $\lambda_j\neq \lambda$ such that
and $0\neq v_\gamma\cdot v_{\lambda_j}=v_\lambda\cdot X_{-\gamma}v_{\lambda_j}$ whence Assertion~(\ref{1tobeproved}) whenever $(\lambda,\alpha)=0$.
Let now $(\lambda,\alpha)\neq 0$. Note that $X_\alpha v_\gamma=0$ in $V$ since $\alpha$ does not belong to the support of $\gamma$.
Then $v_{s_\alpha*\gamma}= X_{-\alpha}^r v_\gamma$ with $r=(\lambda-\gamma,\alpha^\vee)=(\lambda,\alpha^\vee)$.
Further $v_{s_\alpha*\gamma}\not\in\mathfrak g.v_{\lambda}$;
Assertion~(\ref{2tobeproved}) thus holds with $\lambda_i=\lambda$ .

Assume now that $\alpha$ lies in the support of $\gamma$.
Then by Lemma~\ref{rootsupport}, $\gamma-\alpha$ has to be a root; the type $\mathsf F_4$ is ruled out by Lemma~\ref{f4}.
More precisely $\gamma$ is a root of type $\mathsf B_n$ or $\mathsf C_n$.
Further in type $\mathsf B_n$, we can choose $v_\gamma= X_{-\gamma}v_\lambda$
whereas $v_\gamma\in V(\lambda)\setminus \mathfrak g.v_\lambda$ in type $\mathsf C_n$ along with $(\lambda,\alpha)=0$ in both cases.
In the first situation, $v_{s_\alpha*\gamma}=X_{-\gamma-\alpha}v_\lambda$ and there exists $\lambda_i\neq \lambda$ non-orthogonal to $\gamma$.
In type $\mathsf B_n$, we then have $0\neq v_{s_\alpha*\gamma}\cdot v_{\lambda_i}=X_{-\gamma-\alpha}v_{\lambda_i}\cdot v_\lambda$ whence 
Assertion~(\ref{1tobeproved}).
In type $\mathsf C_n$, Assertion~(\ref{1tobeproved}) holds with $\lambda_i=\lambda$.
\end{proof}

\subsection{Proof of Proposition~\ref{non-rigidity}.}

Given $\gamma$ a spherical root, recall the definition of  $S(\gamma)$; \emph{see} Proposition~\ref{non-rigidity}.
Note that $S(\gamma)$ is of cardinality at most $2$; \textsl{see} Table~\ref{sphericalroots}.

In this section, we are given a saturated monoid $\Gamma$ spanned by linearly independent dominant weights $\lambda_i$, $i=1,\ldots,s$.
Recall the definitions of $S^p(\Gamma)$ and $\Sigma(\Gamma)$; \textsl{see} the paragraph before Proposition~\ref{characterisationofH^1}.
We assume also that the support of $\Sigma(\Gamma)$ consists of the whole set of simple roots of the group $G$.

Proposition~\ref{non-rigidity} follows from the lemmas stated below. 

\begin{lemma}~\label{ruleout}
Let $\gamma\in\Sigma(\Gamma)$ be such that $S(\gamma)$ consists of two distinct simple roots $\alpha$ and $\beta$.
If one of the dominant weights $\lambda_i$ is neither orthogonal to $\alpha$ nor to $\beta$ then
all the others are orthogonal to both $\alpha$ and $\beta$.
\end{lemma}

\begin{proof}
The case of $\gamma$ being of type $\mathsf F_4$ is solved in Lemma~\ref{f4}.

Let $[v_\gamma]$ be the $T_\ad$-weight vector in $\left(V/\mathfrak g.v_{\underline\lambda}\right)^{G_{v_{\underline\lambda}}}$ of  weight the given $\gamma$. Choose $v_\gamma\in \oplus V(\lambda_i)_{\lambda_i-\gamma}$.
Let $\lambda$ be one of the given $\lambda_i$'s. Suppose it is neither orthogonal to $\alpha$ nor to $\beta$.

First note that because of saturation, all the $\lambda_i$'s except $\lambda$ have to be orthogonal to (for instance) $\beta$.
If $\gamma$ is orthogonal to one of the simple roots, $\alpha$ and $\beta$, then this simple root has to be $\alpha$.

Take a representative $v_\gamma$ of $[v_\gamma]$ as in Lemma~\ref{weightvector}.

Recall that $\gamma-\delta$ is not a root for all simple roots $\delta$ distinct to $\alpha$ and $\beta$.
Since $X_\delta.v_\gamma\in\mathfrak g.v_{\underline\lambda}$, we have: $X_\delta.v_\gamma=0$ in $V$ for such $\delta$'s.
It follows that $X_\alpha.v_\gamma\neq 0$ or $X_\beta.v_\gamma\neq 0$ in $V$, $v_\gamma$ not being a highest weight vector in $V$.

If $\left(V(\lambda)/\mathfrak g.v_{\lambda}\right)^{G_{v_{\lambda}}}$ is not trivial 
then $(\gamma,\alpha)$ and $(\gamma,\beta)$ are both strictly positive; \textsl{see}~\cite{Js}.
Further one can thus choose the representative $v_\gamma$ in $V(\lambda)$ and such that $X_\beta.v_\gamma\neq 0$ in $V$;
in particular $X_\beta.v_\gamma\in \mathfrak g.v_\lambda$.
For $X_\beta.v_\gamma\in \mathfrak g.v_{\underline\lambda}$, we have
$X_{-\gamma+\beta}.v_{\lambda_i}=0$ for every $\lambda_i\neq \lambda$ hence the $\lambda_i$'s except $\lambda$ have to be orthogonal to $\alpha$.

If $\left(V(\lambda)/\mathfrak g.v_{\lambda}\right)^{G_{v_{\lambda}}}$ is trivial then $v_\gamma=X_{-\gamma}.v_\lambda$.
Since $\gamma-\beta$ is a root, we have: $X_\alpha.v_\gamma=X_{-\gamma+\beta}v_\lambda\neq 0$ in $V$.
We conclude as before that all the $\lambda_i$'s but $\lambda$ are orthogonal to $\alpha$.
\end{proof}

As a straightforward consequence of Lemma~\ref{ruleout}, we obtain the following statement.

\begin{lemma}\label{corollaryruleout}
Consider a spherical root system with a spherical root $\gamma=\alpha_1+\ldots+\alpha_n$ of type $\mathsf A_n$, $n>1$.
Then we may choose one of the dominant weight to be  $\omega_1+\omega_{n}$
if and only if one of the following occurs
\smallbreak
\noindent
{\rm (i)}\enspace
this spherical system is of rank $1$ or
\smallbreak\noindent
{\rm(ii)}\enspace
all the spherical roots $\gamma'$ such that $(\gamma',\gamma)\neq 0$ are of type $\mathsf{A_1}\times\mathsf{A_1}$.
\end{lemma}

\begin{lemma}
Let $\lambda$ be one of the given dominant weights $\lambda_i$.
Let $\alpha$ and $\beta$ be two distinct simple roots.
Assume that $(\lambda_i,\alpha)=(\lambda_i,\beta)=0$ for all $\lambda_i=\lambda$ and that $(\lambda,\alpha).(\lambda,\beta)\neq 0$.
Then one of the following occurs:
\smallbreak
\noindent{\rm (i)}\enspace
$(\gamma,\alpha)=0=(\gamma,\beta)$ for all $\gamma\in\Sigma(\Gamma)$.
\smallbreak
\noindent{\rm (ii)}\enspace
$(\lambda,\delta)=0$ for every simple root $\delta$ distinct to $\alpha$ and $\beta$.
\end{lemma}

\begin{proof}
Since every $T_\ad$-weight $\gamma$ in $\Sigma(\Gamma)$ is an integral sum of the dominant weights, we have:
$(\gamma,\alpha)>0$ if and only if $(\gamma,\beta)>0$ for all $\gamma\in\Sigma(\Gamma)$.
Suppose there exists $\gamma\in\Sigma(\Gamma)$ such that $(\gamma,\alpha)>0$ then $(\gamma,\delta)\leq 0$ for every simple root $\delta$ distinct to $\alpha$ and $\beta$ by Lemma~\ref{rootsupport}.
The second item of the lemma thus follows from the definition of spherical systems and Lemma~\ref{corollaryruleout}.
\end{proof}

\begin{lemma}
Let $\alpha$ be a simple root.
Suppose there exist two distinct dominant weights $\lambda$ and $\lambda'$ among the given $\lambda_i$'s  which are not orthogonal to $\alpha$.
Then $(\gamma,\alpha)=0$ for all $\gamma\in\Sigma(\Gamma)$.
\end{lemma}

\begin{proof}
Take $\gamma\in\Sigma(\Gamma)$ such that $\alpha$ belongs to the support of $\gamma$.
Note that such a $\gamma$ does exist by assumption of the support of $\Sigma(\Gamma)$.
Further $(\gamma,\alpha)\geq 0$ (\textsl{see} Table~\ref{sphericalroots}) and
by Proposition~\ref{rootsupport}, $\gamma-\alpha$ has to be a root. 
By Lemma~\ref{ruleout} we should have: $(\gamma,\alpha)=0$.
But this happens only in type $\mathsf B_n$ or $\mathsf C_n$ (\textsl{see} Table~\ref{sphericalroots}).
In type $\mathsf B_n$, $\alpha$ being extremal, the proposition follows readily.
Consider now the type $\mathsf C_n$ and suppose there exists $\gamma'\in\Sigma(\Gamma)$ such that $(\gamma',\alpha)<0$.
Then there is another simple root, say $\alpha'$, such that $(\gamma',\alpha')<0$ because of saturation.
Such a root $\alpha'$ lies in the support of some $\gamma''\in\Sigma(\Gamma)$.
Further note that $(\gamma'',\alpha)=0$ and $(\gamma'',\alpha')>0$. 
Hence there should exist at least two dominant weights among the $\lambda_i$'s which are not orthogonal to $\alpha'$.
The previous arguments applied to $\gamma''$ and $\alpha'$ lead to a contradiction.
Hence $\gamma'$ has to be orthogonal to $\alpha$.
\end{proof}

\begin{lemma}
If $\lambda_i$ is (up to a scalar) a fundamental weight.
Then $\lambda_i$ is a (up to a scalar) a color in $\Delta(\Gamma)$.
\end{lemma}

\begin{proof}
Let $\alpha$ be the simple root which is not orthogonal to the given weight $\lambda_i$.
Take $\gamma\in\Sigma(\Gamma)$ whose support contains $\alpha$.
Then $\gamma-\alpha$ has to be a root by Lemma~\ref{rootsupport}.
If $\gamma$ is orthogonal to $\alpha$ then $\gamma$ is of type $\mathsf B_n$ or $\mathsf C_n$ whence the lemma.
If $(\gamma,\alpha)>0$ then $\lambda_i$ is a color except if $\gamma$ is of type $\mathsf A_1\times\mathsf A_1$.
In the latter case, observe that $\gamma\in\Sigma(\Gamma)$ whenever one of the $\lambda_j$'s is neither orthogonal to $\alpha$ nor  $\alpha'$ for $\gamma=\alpha+\alpha'$. But because of saturation this is not possible.
\end{proof}

\begin{proof}[Proof of Proposition~\ref{converse}]
Let $(S^p,\Sigma)$ be a spherical system of $G$ and $\Delta$ be its set of colors.
Recall the definition of $\Gamma$ stated in Proposition~\ref{converse} and let $V$ be the corresponding $G$-module.
Take $\gamma\in \Sigma$. 
We shall prove that $\gamma\in\Sigma(\Gamma)$, namely that $\gamma$ is a $T_\ad$-weight of $\left(V/\mathfrak g.v_{\underline\lambda}\right)^{G_{v_{\underline\lambda}}}$.

By the characterization of the monoid $\Gamma$ along with Lemma~\ref{weightvector} (\textsl{see} also Lemma~\ref{f4}), one gets that in most cases, there exists a unique dominant weight among the given $\lambda_i$'s which is not orthogonal to the given $\gamma$.
Proposition~\ref{converse} thus follows from \cite{Js}.

Suppose now that there are more than one dominant weight which is not orthogonal to $\gamma$.
Then by definition, this occurs whenever $S(\gamma)$ is a $2$-set. 
By Table~\ref{sphericalroots}, the weight $\gamma$ has to be a root. 
One thus checks easily that $X_{-\gamma}v_\lambda$ is indeed a $T_\ad$-weightvector of $\left(V(\underline\lambda)/\mathfrak g.v_{\underline\lambda}\right)^{G_{v_{\underline\lambda}}}$.

\end{proof}

\end{document}